\documentclass{llncs}

\usepackage{amssymb}
\usepackage[english]{babel}
\usepackage{t1enc}
\usepackage[latin2]{inputenc}
\usepackage{epsfig}

\newtheorem{conj}{Conjecture}

\newcommand{\qedbox}{\hfill \ensuremath{\Box}}

\begin{document}
\title{Density-based group testing}

\author{D\'aniel Gerbner\inst{1}\fnmsep\thanks{Research supported in part by the Hungarian NSF,
under contract NK 78439}\and Bal\'azs
Keszegh\inst{1}\fnmsep$^\star$\and D\"om\"ot\"or
P\'alv\"olgyi\inst{2}\fnmsep\thanks{Research supported by
Hungarian NSF, grant number: OTKA CNK-77780} \and G\'abor
Wiener\inst{3}\fnmsep\thanks{Supported in part by the Hungarian
National Research Fund and by the National Office for Research and
Technology (Grant Number OTKA 67651)}}

\institute{R\'enyi Institute of Mathematics, Hungarian Academy of
Sciences \,
\email{\{gerbner.daniel,keszegh.balazs\}@renyi.mta.hu}\and
Department of Computer Science, E\"otv\"os University
\,\email{dom@cs.elte.hu} \and Department of Computer Science and
Information Theory, Budapest University of Technology and
Economics \, \email{wiener@cs.bme.hu}}

\maketitle

\begin{abstract}
\noindent In this paper we study a new, generalized version of the
well-known group testing problem. In the classical model of group
testing we are given $n$ objects, some of which are considered to
be defective. We can test certain subsets of the objects whether
they contain at least one defective element.  The goal is usually
to find all defectives using as few tests as possible. In our
model the presence of defective elements in a test set $Q$ can be
recognized if and only if their number is large enough compared to
the size of $Q$. More precisely for a test $Q$ the answer is {\sc
yes}  if and only if there are at least $\alpha |Q|$ defective
elements in $Q$ for some fixed $\alpha$.
\end{abstract}

\smallskip


\smallskip


\section{Introduction}\label{intro}

The concept of group testing was developed in the middle of the previous century. Dorfman, a Swiss physician intended to
test blood samples of millions of soldiers during World War II in order to find those who were
infected by syphilis. His key idea was to test more blood samples
at the same time and learn whether at least one of them are infected \cite{dorfman}. Some fifteen years later R\'enyi developed a theory of search in order to find which electrical part of his car went wrong. In his model -- contrary to Dorfman's one -- not all of the subsets of the possible defectives (electric parts) could be tested \cite{K3}.

Group testing has now a wide variety of applications in areas like DNA screening, mobile networks, software and hardware testing.

In the classical model we have an underlying set $[n]=\{1,\dots,n\}$ and we suppose that there
may be some defective elements in this set. We can test all subsets of
$[n]$ whether they contain at least one defective element. The goal is
to find all defectives using as few tests as possible. One can easily see that in this generality
the best solution is to test every set of size 1. Usually we have some additional information like the exact number of defectives (or some bounds on this number) and it is also frequent that we do not have to find all defectives just some of them or even just to tell something about them.

In the case when we have to find a single defective it is well-known that the information
theoretic lower bound is sharp: the number of questions
needed in the worst case is $\lceil \log n \rceil$, which can be achieved by binary search.

Another well-known version of the problem is when the maximum size of a test
is bounded. (Motivated by the idea that too large tests are not supposed to be reliable,
because a small number of defectives may not be recognized
there). This version can be solved easily in the adaptive case, but is much more difficult in the non-adaptive case. This latter version was first posed by R\'enyi. Katona \cite{Katona} gave an algorithm to find the exact solution to R\'enyi's problem and he also proved the best known lower bound on the number of queries needed. The best known upper bound is due to Wegener \cite{wegener}.

In this paper we assume that the presence of defective elements in a test set $Q$ can be recognized
if and only if their number is large enough compared to the size of
$Q$. More precisely for a test $ Q \subseteq [n]$ the
answer is {\sc yes}  if and only if there are at least $\alpha |Q|$
defective elements in $Q$. Our goal is to find at least $m$
defective elements using tests of this kind.

\begin{definition} Let $g(n,k,\alpha,m)$ be the least number of questions
needed in this setting, i.e. to find $m$ defective elements in an
underlying set of size $n$ which contains at least $k$ defective
elements, where the answer is {\sc yes} for a question $ Q \subseteq [n]$
if and only if there are at least $\alpha |Q|$ defective elements
in $Q$.
\end{definition}

We suppose throughout the whole paper that $1\le m\le k$ and
$0<\alpha<1$. Let $a=\lfloor \frac{1}{\alpha}\rfloor$, that is,
$a$ is the largest size of a set where the answer {\sc no}  has
the usual meaning, namely that there are no defective elements in
the set. It is obvious that if a set of size greater than
$k/\alpha$ is asked then the answer is automatically {\sc no}, so
we will suppose that question sets has size at most $k/\alpha$.
All logarithms appearing in the paper are binary.

It is worth mentioning that a similar idea appears in a paper by Damaschke \cite{damaschke}
and a follow-up paper by De Bonis, Gargano, and Vaccaro
\cite{debonis}. Since their motivation is to study the concentration of
liquids, their model deals with many specific properties
arising in this special case and they are
interested in the number of merging operations or the number of
tubes needed in addition to the number of tests.

If $k=m=1$, then the problem is basically the same as the usual
setting with the additional property that the question sets can
have size at most $a$: this is the above mentioned problem of
R\'enyi. As we have mentioned, finding the optimal non-adaptive
algorithm, or even just good bounds is really hard even in this
simplest case of our model, thus in this paper we deal only with
adaptive algorithms.

In the next section we give some upper and lower bounds as well
as some conjectures depending on the choices of $n$, $k$, $\alpha$,
and $m$. In the third section we prove our main theorem, which
gives a general lower and a general upper bound, differing only by a constant
depending only on $k$. In the fourth section we consider some
related questions and open problems.

\section{Upper and lower bounds}\label{mbarmi}

First of all it is worth examining how binary search, the most basic algorithm of search theory works in our setting. It is easy to see that it does not work in general, not even for $m=1$. If (say) $ k=2$ and $ \alpha = 0.1$, then question sets have at most 20 elements (recall that we supposed that there are no queries containing more than $k/\alpha$ elements, since they give no information at all, because the answer for them is always {\sc no}), thus if $n$ is big, we cannot perform a binary search.

However, if $k\geq n\alpha$, then binary search can be used.

\begin{theorem}\label{gyors} If $\alpha \le k/n$, then $g(n,k,\alpha,m) \le \lceil\log n\rceil
+c$, where $c$ depends only on $\alpha$ and $m$, moreover if
$m=1$, then $c=0$.

\end{theorem}

\begin{proof} We show that binary search can be used to find $m$ defectives. That is,
first we ask a set $F$ of size $\lfloor n/2 \rfloor$ and then the
underlying set is substituted by $F$ if the answer is {\sc yes}
and by $\overline{F}$ if the answer is {\sc no}. We iterate this
process until the size of the underlying set is at most
$2m/\alpha$. Now we check that the condition $\alpha \le k/n$
remains true after each step. Let $n'=\lfloor n/2 \rfloor $ be the
size of the new underlying set and $k'$ be the number of
defectives there. If the answer was {\sc yes}  , then $k' \ge
\alpha n'$, thus $\alpha \le k'/n'$. If the answer was {\sc no},
then there are at least $k- \lceil \alpha n' \rceil +1 $
defectives in the new underlying set, that is $k' \ge k- \lceil
\alpha n' \rceil +1  \ge \alpha n- \lceil \alpha n'  \rceil +1 \ge
\alpha n'$, thus $\alpha \le k'/n'$ again.

Now if $m=1$ we simply continue the binary search until we
find a defective element, altogether using at most $\lceil \log
n\rceil$ questions.

If $m>1$, then we can find $m$ defectives in the last underlying set using at most
$c:=\max_{n' \le 2m/\alpha} g(n',m, \alpha, m)$ further queries.

(Notice that since the size of the last underlying set is greater than $m/\alpha$, it contains at least $m$ defectives.)
This number $c$ does not depend on $k$, just on $\alpha$ and $m$ and it is obvious that
we used at most $\lceil\log n\rceil +c$ queries altogether.
\qedbox \end{proof}

This theorem has an easy, yet very important corollary. If the answer for a question $A$ is {\sc yes}, then there are
at least $\alpha|A|$ defective elements in $A$. If $\alpha|A| \ge
m$, then we can find $m$ of these defectives using $g(|A|,\alpha|A|,\alpha,m)
\le \log |A| +c$ questions, where $c$ depends only on $\alpha$ and
$m$. Basically it means that whenever we obtain a {\sc yes} answer, we can
finish the algorithm quickly.

The proof of Theorem \ref{gyors} is based on the fact that if the ratio of the defective elements $k/n$ is at least
$\alpha$, then this condition always remains true during binary search. If $k/n < \alpha$, then this trick does not work, however if the difference between $k/n$ and $\alpha$ is small, a similar result can be proved for $m=1$. Recall that $a= \lfloor 1/\alpha \rfloor $.

\begin{theorem}\label{felezgetes} If
$k \ge \frac{n}{a}- \lfloor\log \frac{n}{a}\rfloor-1$ and $k\ge 1$, then $g(n,k,\alpha,1) \le
\lceil\log n\rceil +1 $.
\end{theorem}

The proof of the theorem is based on the following lemmas.

\begin{lemma}\label{felezgetes1}
Let $t\geq 0$ be an integer. Then  $g(2^ta,2^t-t,\alpha,1) \le
 t + \lceil \log a \rceil $.
\end{lemma}
\begin{proof}
We use induction on $t$. For $t=0$ and $t=1$ the proposition is
true, since we can perform a binary search on $a$ or $2a$ elements
(by asking sets of size at most $a$ we learn whether they contain
a defective element). Suppose now that the proposition holds for
$t$, we have to prove it for $t+1$. That is, we have an underlying
set of size $2^{t+1}a$ containing at least $2^{t+1}-t-1$
defectives. Our first query is a set $A$ of size $2^ta$. If the
answer is {\sc yes}  , then we can continue with binary search. If
the answer is {\sc no}, then there are less than $\alpha 2^ta \leq
2^t $ defectives in $A$, therefore there are at least
$2^{t+1}-t-1-2^t+1=2^t-t$ defectives in $\overline{A}$. By the
induction hypothesis $g(2^ta,2^t-t,\alpha,1) \le t + \lceil \log a
\rceil $, thus $g(2^{t+1}a,2^{t+1}-t-1,\alpha,1) \le t+1+ \lceil
\log a \rceil $ follows, finishing the proof of the lemma.
\qedbox \end{proof}

\begin{lemma}\label{felezgetes2}
Let $t\geq 2$ be an integer. Then  $g(2^ta,2^t-t-1,\alpha,1) \le
 t + \lceil \log a \rceil+1 $.
\end{lemma}
\begin{proof} Let us start with asking three disjoint sets, each of cardinality $2^{t-2}a$. If the answer to any of these is {\sc yes}, then we can continue with binary search, using $t-2+\lceil \log a\rceil$ additional questions. If all three answers are {\sc no}, then there are at least $2^t-t-1-3(2^{t-2}-1)=2^{t-2}-(t-2)$ defectives among the remaining $2^{t-2}a$ elements, hence we can apply Lemma \ref{felezgetes1}.
\qedbox \end{proof}

\begin{proof}[of Theorem \ref{felezgetes}]
Let us suppose $n>2a$ (otherwise binary search works) and let $t=\lfloor \log \frac{n}{a} \rfloor$, $r=n-2^ta$. We have an underlying set of size $n=2^{t}a+r$ containing at least $\frac{n}{a}-\lfloor\log \frac{n}{a}\rfloor-1 $ defectives.
If $r=0$, then by Lemma \ref{felezgetes2} we are done. Otherwise let the first query $A$ contain $r$ elements. A positive answer allows us to find a defective element by binary search on $A$ using altogether at most
$\lceil \log n \rceil +1$ questions (actually, at most $\lceil \log n \rceil $ questions, because $r\le n/2$). If the answer is negative then the new underlying set contains $2^ta$ elements, of which more than $ \frac{n}{a}-\lfloor\log \frac{n}{a}\rfloor - \alpha r -1 = 2^t + r/a -\alpha r -\lfloor\log \frac{n}{a}\rfloor -1 \ge 2^t - \lfloor\log \frac{n}{a}\rfloor-1 $ are defective. Since $\lfloor \log \frac{n}{a} \rfloor =t$, the number of defectives is at least $2^t-t$, thus by Lemma \ref{felezgetes1} we need at most $t + \lceil \log a \rceil $ more queries to find a defective element, thus altogether we used at most $t + 1 + \lceil \log a \rceil \le \lceil\log n\rceil +1 $ queries, from which the theorem follows.
\qedbox \end{proof}

\smallskip

One might think that binary search is the best algorithm to find one defective if it can be used (i.e. for $k\geq n\alpha$). A counterexample for $k$ really big is easy to give: if $k=n$ then we do not need any queries and for $m=1, k=n-1$ we need just one query. It is somewhat more surprising that $g(n,\alpha n, \alpha, 1) \ge \lceil \log n \rceil $ is not necessarily true.

For example, the case $n=10, k=4, \alpha = 0.4, m=1$ can be solved using 3 queries: first we ask a set $A$ of size 4. If the answer is {\sc yes}, we can perform a binary search on $A$, if the answer is {\sc no}  then there are at least 3 defectives among the remaining 6 elements and now we ask a set $B$ of size 2. If the answer is {\sc yes}   then we perform a binary search on $B$, otherwise there are at least 3 defectives among the remaining 4 elements, so one query (of size 1) is sufficient to find a defective.
However, a somewhat weaker lower bound can be proved:

\begin{theorem}\label{keszegh} $g(n,k,\alpha,m) \ge \lceil \log(n-k+1) \rceil $.
\end{theorem}

We prove the stronger statement that even if one can use any kind of yes-no questions, still at least $\lceil \log(n-k+1) \rceil $ questions are needed. This is a slight generalization of the information theoretic lower bound.

\begin{theorem}\label{domotor} To find one of $k$ defective elements from
a set of size $n$, one needs $\lceil \log(n-k+1) \rceil $ yes-no
questions in the worst case and this is sharp.
\end{theorem}

\begin{proof}  Suppose there is an algorithm that uses at most $q$
questions. The number of sequences of answers obtained is at most $2^q$, thus the
number of different elements selected by the algorithm as the output is also at most $2^q$.
This means that $n-2^q\le k-1$, otherwise it would be
possible that all $k$ defective elements are among those ones that were
not selected. Thus $q \ge \lceil \log(n-k+1) \rceil $ indeed.

Sharpness follows easily from the simple algorithm that puts $k-1$ elements
aside and runs a binary search on the rest.
\qedbox \end{proof}

Theorem \ref{keszegh} is an immediate consequence of Theorem \ref{domotor}, but this is not true for the sharpness of
the result. However, Theorem \ref{keszegh} is also sharp: if $\alpha \le \frac{2}{n-k+1}$, then we can run a binary search on any $n-k+1$ of the elements to find a defective.

We have seen in Theorem \ref{gyors} that if $n \leq k/\alpha$, then binary search works (with some additional constant number of questions if $m>1$). On the other hand, if
$n$ goes to infinity (with $k$ and $\alpha$ fixed), then the best algorithm is linear.

\begin{theorem}\label{approxi} For any $k$, $\alpha$, $m$

\[\frac{n}{a}+c_1 \le g(n,k,\alpha,m)\le \frac{n}{a}+c_2,\]

where $c_1$ and $c_2$ depend only on $k$, $\alpha$, and $m$.

\end{theorem}

\begin{proof} Upper bound: first we partition the underlying set into
$\lfloor \frac{n}{a}\rfloor$ $a$-element sets and possibly one additional set of
less than $a$ elements. We ask each of these sets (at most $\lfloor
\frac{n}{a}\rfloor+1$ questions). Then we choose $m$ sets for which we obtained a {\sc yes}   answer
 (or if there are less than $m$ such sets, then we choose all
of them). We ask every element one by one in these sets (at most $ma$
questions). One can easily see that we find at least $m$ defective
elements, using at most $\lfloor
\frac{n}{a}\rfloor+ma+1$ questions.

Lower bound: We use a simple adversary's strategy: suppose all the answers are {\sc no}  and there are $m$ elements identified as defectives.
Let us denote the family of sets that were asked by
$\mathcal{F}$. It is obvious that those sets of $\mathcal{F}$ that have size at most $a$ contain no defective elements.
Suppose there are $i$ such sets. We use induction on $i$. There are $n'
\ge n-ia$ elements not contained in these sets and we should
prove that at least $\frac{n}{a}+c_1-i\le \frac{n'}{a}+c_1$ other
questions are needed. Hence by the induction it is enough to prove
the case $i=0$.

Suppose $i=0$. If there is a set $A$ of size $k+1$, such that
$|A\cap F|\le 1$ for all $F\in \mathcal{F}$, then any $k$-element
subset of $|A|$ can be the set of the defective elements. In this case any element can be non-defective, a contradiction. Thus for every set $A$ of size $k+1$  there exists a set $F\in \mathcal{F}$, such that $|A \cap F| \ge 2$.

Let $b=\lfloor \frac{k}{\alpha}\rfloor$. We know that every set of $\mathcal{F}$ has size at
most $b$. Then a given
$F\in\mathcal{F}$ intersects at most $\sum_{j=2}^{k+1}{b \choose
j}{n-b \choose k+1-j}$ $(k+1)$-element sets in at least two points.
This number is $O(n^{k-1})$, and there are $\Omega(n^{k+1})$ sets
of size $k+1$, hence $|\mathcal{F}|=\Omega(n^{2})$ is needed.

It follows easily that there is an $n_0$, such that if $n>n_0$, then $|\mathcal{F}| \ge \frac{n}{a}$. Now let $c_1=-n_0/a$. If $n>n_0$ then
$|\mathcal{F}| \ge \frac{n}{a} \ge \frac{n}{a}+c_1$, while if $n \le n_0$ then
$|\mathcal{F}| \ge 0 \ge \frac{n}{a}+c_1$, thus the number of queries is at least  $\frac{n}{a}+c_1$, finishing the proof.
\qedbox \end{proof}

\noindent \emph{Remark.}  The theorem easily follows from Theorem \ref{maintheorem}, it is included here because of the much simpler proof.

\

It is easy to give a better upper bound for $m=1$.

\begin{theorem}\label{uj1} Suppose $k+\log k +1 \leq \lceil \frac{n}{a} \rceil$. Then
\[g(n,k,\alpha,1)\le \left\lceil \frac{n}{a} \right\rceil -k + \lceil \log a \rceil .\]
\end{theorem}

\begin{proof}
First we ask a set $X$ of size $ka$. If the answer is
{\sc yes}, then we can find a defective element in $\lceil \log ka \rceil$ steps
by Theorem \ref{gyors}. In this case the number of questions used is at most $1+\lceil \log ka \rceil = 1 + \lceil \log k + \log a \rceil \leq 1 + \lceil \log k \rceil + \lceil \log a \rceil \leq \lceil \frac{n}{a} \rceil -k + \lceil \log a \rceil$, where the last inequality follows from the condition of the theorem.

If the answer is {\sc no}, then we know that there are
at most $k-1$ defectives in $X$, so we have at least one defective in $\overline{X}$. Continue the algorithm by asking disjoint subsets of $X$ of size $a$, until the answer
is {\sc yes} or we have at most $2a$ elements not yet asked. In these cases
using at most $\lceil\log 2a\rceil $ questions we can easily find a
defective element, thus the total number of questions used is at most $1+  \lceil \frac{n-ka-2a}{a} \rceil +\lceil \log 2a \rceil = 1 +  \lceil \frac{n}{a} \rceil -k -2 +\lceil \log a \rceil +1 = \lceil \frac{n}{a} \rceil -k +\lceil \log a \rceil$, finishing the proof.
\qedbox \end{proof}

Note that if the condition of Theorem \ref{uj1} does not hold
(that is, $k+\log k +1 > \lceil \frac{n}{a} \rceil$), then $k\ge
\frac{n}{a}- \lfloor\log \frac{n}{a}\rfloor-1$, hence $\lceil \log
n \rceil +1$ questions are enough by Theorem \ref{felezgetes}.

The exact values of $g(n,k,\alpha,m)$ is hard to find, even for $m=1$. The algorithm used in the proof of Theorem \ref{uj1} seems to be optimal for $m=1$ if $k+\log k +1 \leq \lceil \frac{n}{a} \rceil $. However, counterexamples with $ 1/\alpha$ not an integer are easy to find (consider i.e. $n=24$, $k=2$, $\alpha = \frac{2}{11}$).

\begin{conj} \label{sej1} If $\frac{1}{\alpha}$ is an integer and $k+\log k +1 \leq \lceil \frac{n}{a} \rceil $,
then the algorithm used in the proof of Theorem \ref{uj1} is optimal for $m=1$.
\end{conj}

It is easy to see that Conjecture \ref{sej1} is true for $k=1$. For other values of $k$ it would follow from the next, more general conjecture.

\begin{conj}\label{integer} If $\frac{1}{\alpha}$ is an integer, then
$g(n,k,\alpha,1) \le g(n,k+1,\alpha,1)+1$.
\end{conj}

Obviously, Conjecture \ref{integer} also fails if $1/\alpha$ is not an
integer. One can see for example that $g(24,1,2/11,1)=7$ and $g(24,2,2/11,1)=5$.

\section{The main theorem}

In this section we prove a lower and an upper bound differing only by a constant depending only on $k$.
For the lower bound we need the following simple generalization of the information
theoretic lower bound.

\begin{proposition}\label{trivi2} Suppose we are given $p$ sets $A_1, \ldots , A_p$ of size at least $n$,
each one containing at least one defective and an additional set $A_0$ of arbitrary size containing no
defectives. Let $m \le p$. Then the number of
questions needed to find at least $m$ defectives is at
least $\lceil m \log n \rceil$.
\end{proposition}

\begin{proof}

Suppose that we are given the additional information that every set $A_i$ ($i \ge 1$)
contains exactly one defective element. Now we use the
information theoretic lower bound: there are $\prod_{i=1}^p
|A_i|$ possibilities for the distribution of the defective elements at the beginning,
and at most $\prod_{i=1}^{p-m}|A_{j_i}|$ at the end (suppose we have found defective elements in every set
$A_i$ except in  $A_{j_1}, \dots, A_{j_{p-m}}$), thus if we used $l$
queries, then $2^l \ge
n^m$, from which the proposition follows.
\qedbox \end{proof}

Now we formulate the main theorem of the paper.

\begin{theorem} \label{maintheorem} For any $k$, $\alpha$, $m$
\[
\frac{n}{a} +m\log a-c_1(k)\le g(n,k,\alpha,m)\le \frac{n}{a} +m\log a+c_2(k),
\]
where $c_1(k)$ and $c_2(k)$ depend only on $k$.
\end{theorem}

\begin{proof}
First we give an algorithm that uses at most $\frac{n}{a} +m\log a+c_2(k)$ queries, proving the upper bound.
In the first part of the procedure we ask disjoint sets $A_1, A_2, \ldots , A_r$ of size $a$ until either there were $m$ {\sc yes}   answers or there are no more elements left. In this way we ask at most $\lceil \frac{n}{a} \rceil $ questions.

Suppose we obtained {\sc yes}   answers for the sets $A_1, A_2, \ldots A_{m_1}$ and {\sc no}  answers for the sets $A_{m_1+1}, \ldots , A_{r}$. If $m_1 \ge m$, then in the second part of the procedure we use binary search in the sets $A_1, A_2, \ldots , A_m$ in order to find one defective element in each of them. For this we need $m \lceil \log a \rceil $ more questions.

If $m_1<m$, then first we use binary search in the sets $A_1, A_2, \ldots , A_{m_1}$ in order to find defective elements $a_1 \in A_1, a_2\in A_2, \ldots , a_{m_1} \in A_{m_1} $. Then we iterate the whole process using $S_1 = \cup_{i=1}^{m_1} A_i\setminus \{ a_i \}$ as an underlying set, that is we ask disjoint sets $B_1, B_2, \ldots , B_t$
of size $a$ until either we obtain $m-m_1$ {\sc yes}   answers or there are no more elements left. Suppose we obtained {\sc yes}   answers for the sets $B_1, B_2, \ldots B_{m_2}$ and {\sc no}  answers for the sets $A_{m_2+1}, \ldots , A_{t}$. If $m_2 \ge m-m_1$, then in the second part of the procedure we use binary search in the sets $B_1, B_2, \ldots , B_{m-m_1}$ in order to find one defective element in each of them, while if $m_2<m-m_1$, then first we use binary search in the sets $B_1, B_2, \ldots , B_{m_2}$ in order to find defective elements $b_1 \in B_1, b_2\in B_2, \ldots , b_{m_2} \in A_{m_2} $ and continue the process using $S_2 = \cup_{i=1}^{m_2} B_i\setminus \{ b_i \}$ as an underlying set, and so on, until we find $m=m_1+m_2+\ldots+m_j$ defective elements. Note that $m_i\ge 1, \; \forall i \le j$, since $k\ge m$.
We have two types of queries: queries of size $a$ and queries of size less than $a$ (used in the binary searches).
The number of questions of size $a$ is at most $\lceil \frac{n}{a} \rceil $ in the first part and at most
$m_1+m_2+\ldots+m_{j-1}< m\le k$ in the second part. The total number of queries of size less than $a$ is at most
$m \lceil \log a \rceil $, thus the total number of queries is at most $\lceil \frac{n}{a} \rceil + m \lceil \log a \rceil + k$, proving the upper bound.

To prove the lower bound we need the following purely set-theoretic lemma.

\begin{lemma} \label{domlemma} Let $k,l,a$ be arbitrary positive integers and $\beta > 1$.
Let now $\cal H$ be a set system on an underlying set $S$ of size $c(k,l,\beta)\cdot a=k\beta (2^{kl}-1)a$, such that every set of $\cal H$ has size at most $\beta a$ and every element of $S$ is contained in at most $l$ sets of $\cal H$. Then we can select $k$ disjoint subsets of $S$ (called heaps) $K_1,K_2 , \ldots ,K_k$ of size $\beta a$, such that every set of $\cal{H}$ intersects at most one heap.
\end{lemma}

\begin{proof}
Let us partition the underlying set into $k$ heaps of size $\beta
a(2^{kl}-1)$ in an arbitrary way. Now we execute the following
procedure at most $kl-1$ times, eventually obtaining $k$ heaps
satisfying the required conditions. In each iteration we make sure
that the members of a subfamily $\cal H'$ of $\cal H$ will intersect
at most one heap at the end.

In each iteration we do the following. We build the subfamily ${\cal H'}
\subseteq {\cal H}$ by starting from the empty subfamily and adding an
arbitrary set of $\cal H$ to our subfamily until there exists a heap
$K_i$ such that $|K_i \cap \cup_{H\in {\cal H'}} H | \ge |K_i|/2$, that
is $K_i$ is at least half covered by ${\cal H'}$. We call $K_i$ the
selected heap. If the half of several heaps gets
covered in the same step, then we select one where the difference of the number of covered elements and the half of the size of the heap is maximum.

Now we keep the covered part of the selected heap and keep the
uncovered part of the other heaps  and throw away the other elements. We also throw away the
sets of the subfamily ${\cal H'}$ from our family ${\cal H}$, as we
already made sure that the members of ${\cal H'}$  will not intersect
more than one heap at the end. In this way we obtain smaller heaps but
we only have to deal with the family ${\cal H} \setminus {\cal H'}$.

We prove by induction that after $s$ iterations all heaps have size at
least $\beta a(2^{kl-s}-1)$. This trivially holds for $s=0$. By
the induction hypothesis, the heaps had size at least
$\beta a(2^{kl-s+1}-1)$ before the $s$th iteration step. After the $s$th step the new size of
the selected heap $K$ is at least $|K|/2\ge \beta
a(2^{kl-s+1}-1)/2\ge \beta a(2^{kl-s}-1)$. Now we turn our attention to the unselected heaps.
Suppose the set we added last to $\cal H'$ is the set $I$. Clearly,
$|K_j \cap \cup_{H\in {\cal H'}\setminus\{I\}} H | \le |K_j|/2$ for
all $j$. Let $K$ be the selected heap and $K_i$ be an arbitrary unselected heap.  Now by the choice of $K$ we have $|K_i \cap \cup_{H\in {\cal H'}} H | \le |K_i|/2 + |I|/2$, otherwise $|K_i \cap \cup_{H\in {\cal H'}} H | + |K\cap \cup_{H\in {\cal H'}} H | > |K_i|/2 + |K|/2 + |I|$, which is impossible, since $|K_i \cap \cup_{H\in {\cal H'}} H | + |K\cap \cup_{H\in {\cal H'}} H | = | ((K_i\cup K) \cap \cup_{H\in {\cal H'}\setminus\{I\}} H ) \cup ((K_i\cup K) \cap I )| \leq  |K_i|/2 + |K|/2 + |I|$.

Now since $|I| \leq \beta a$, the new size of the unselected heap $K_i$ is
$|K_i'|=|K_i \setminus \cup_{H\in {\cal H'}} H | \ge |K_i|/2-\beta
a/2 \ge \beta a(2^{kl-s+1}-1)/2-\beta a/2\ge \beta a(2^{kl-s}-1)$, finishing the proof by induction.

Now in each iteration we delete a family that covers the
selected heap, thus any heap can be selected at most $l$ times, since
every element is contained in at most $l$ sets.
After $kl-1$ iterations the size of an arbitrary heap will be still at
least $\beta a$. Furthermore, all but one heaps were selected exactly
$l$ times, thus any remaining set of $\cal H$ can only intersect the last heap. That is, heaps
at this point satisfy the required condition for all sets of $\cal H$.

If we can iterate the process at most $kl-2$ times, then after the last possible
iteration more than half of any heap is not covered by the union of the
remaining sets. Deleting the covered elements from each heap we
obtain heaps of size at least $\beta a$ that satisfy the condition.
\qedbox \end{proof}

Now we are in a position to prove the lower bound of Theorem \ref{maintheorem}. We use the adversary method, i.e. we give a strategy to the adversary that forces the questioner to ask at least $\frac{n}{a} +m\log a-c_1(k)$ questions to find $m$ defective elements.

Recall that all questions have size at most $\lfloor k/\alpha \rfloor$ and now the adversary gives the additional information that there are exactly $k$ defective elements.

During the procedure, the adversary maintains weights on the elements. At the beginning all elements have weight $0$. Let us denote the set of the possible defective elements by $S'$. At the beginning $S'=S$.
At each question $A$ the strategy determines the answer and also adds appropriate weights to the elements of $A$.
If a question $A$ is of size at most $a=\lfloor 1/\alpha\rfloor$, then the answer is {\sc no}  and weight $1$ is given to all elements of $A$. If $|A|>a$, the answer is still {\sc no}  and weight $a/\lfloor k/\alpha \rfloor$ is given to the elements of $A$. Thus after some $r$ questions the sum of the weights is at most $ra$.
If an element reaches weight $1$, then the adversary says that it is not defective, and the element is deleted from $S'$. The adversary does that until there are still $ca$ elements in $S'$ but in the next step $S'$ would become smaller than this threshold (the exact value of $c$ will be determined later). Up to this point the number of elements thrown away is at least $n-ca-\lfloor k/\alpha \rfloor$, thus the number of queries is at least $\frac{n}{a}-c-\lfloor k/\alpha \rfloor/a\ge \frac{n}{a}-c-k$.

Let the set system $\cal F$ consist of the sets that were asked up to this point and let ${\cal F}' = \{ F\cap S' \; | \; F\in {\cal F} , |F|> a \} $.

The following observations are easy to check.

\begin{lemma} \
\begin{itemize}
\item $|S'| \ge ca$.
\item Every set $F \in \cal F'$ has size at most $\lfloor k/\alpha \rfloor\le k(a+1)\le 2ka$ 
\item Every element of $S'$ is contained in at most $\lfloor k/\alpha \rfloor/a\le k(1+1/a)\le 2k$ sets of $\cal F'$.
\item Every $k$-set that intersects each $F\in \cal F'$ in at most one element is a possible set of defective elements.
\end{itemize}
\end{lemma}

Now let $l:=2k$, $\beta:=2k$, and $c:=c(k,l,\beta )=k\beta ( 2^{kl}-1)=2k^2(2^{2k^2}-1)$.
By the observations above, we can apply Lemma \ref{domlemma} with $\cal H=\cal F'$. The lemma guarantees the existence of heaps $K_1,K_2, \ldots , K_k$ of size $\beta a\ge a$, such that every transversal of the $K_i$'s is a possible $k$-set of defective elements. Now by applying Proposition \ref{trivi2} with $A_i=K_i$ and $A_0=S\setminus S'$, we obtain that the questioner needs to ask at least $\lceil m\log a \rceil$ more queries to find $m$ defective elements.

Altogether the questioner had to use at least $\frac{n}{a}-c-k+m\log a$ queries, which proves the lower bound, since the number $c$ depends only on $k$ (the constant in the theorem is $c_1(k)=c+k$).
\qedbox \end{proof}

The constant in the lower bound is quite large, by a more careful
analysis one might obtain a better one. For example, we could redefine
the weights, such that we give weight $a/|A|$ to the elements of
$A$, thus still distributing weight at most $a$ per asked set.

It is also worth observing that if $1/\alpha$ is an integer, then we can use
Lemma \ref{domlemma} with $l=\beta=k$, instead of $l=\beta=2k$. This way one can prove
stronger results for small values of $k$ and $m$ if $1/\alpha$ is
an integer. We demonstrate it for $k=2$ in the next section. The following claim
is easy to check.

\begin{claim}\label{c1claim}
Let $\cal H$ be a set system on an underlying set $S$ of size $3a$, consisting of disjoint sets of size at most $2a$. Then we can select $2$ disjoint subsets of $S$ (called heaps) $K_1,K_2$ of size at least $a$, such that every set of $\cal{H}$ intersects at most one heap.
\end{claim}

\section{The case $k=2$}

In this section we determine the exact value $g(n,2,\alpha,1)$.
Let $\delta=\lfloor 2 \{ \frac{1}{\alpha} \} \rfloor $, where $\{ x \}$ denotes the fractional part of $x$.

Consider the following algorithm W, where $n$ denotes the number of remaining elements:

If $n\le 2^{\lceil\log a\rceil}+1$, we ask a question of size $\lfloor n/2\rfloor\le a$, then depending on the answer we continue in the part that contains at least one defective element, and find that with binary search.

If $2^{\lceil\log a\rceil} + 2\le n\le 2^{\lceil\log
a\rceil+1}+1$, then we ask a question of size $2^{\lceil\log
a\rceil}+1$ (this falls between $a$ and $2a+1$). If the answer is
{\sc yes}, we put an element aside and continue with the remaining
elements of the set we asked, otherwise we continue with the
elements not in the set we asked. This way independent of whether
we got a {\sc yes} or {\sc no} answer, we have at most
$2^{\lceil\log a\rceil}$ elements with at least one defective,
hence we can apply binary search.

If $2^{\lceil\log a\rceil+1}+2\le n\le 3a + \delta + 2^{\lceil\log
a\rceil} $, then first we ask a question of size $2a+\delta$. If
the answer is {\sc yes}, we put an element aside and continue with
the remaining elements of the set we asked, otherwise we continue
with the elements not in the set we asked. This way independent of
whether we got a {\sc yes} or {\sc no} answer, we have at most
$2^{\lceil\log a\rceil} + a$ elements with at least one defective.
We continue with a set of size $a$, and after that we can finish
with binary search.

If $n\ge 3a + \delta + 2^{\lceil\log a\rceil}+1$, then we ask a
question of size $a$. If the answer is {\sc no}, we proceed as
above. If the answer is {\sc yes}, we can find a defective element
with at most $\lceil\log a\rceil$ further questions.

Counting the number of questions used in each case, we can conclude.

\begin{claim}\label{egyes} If $n\le 3a + \delta + 2^{\lceil\log
a\rceil}$, then algorithm W takes only $\lceil\log (n-1)\rceil$ questions,
thus according to Theorem \ref{domotor} it is optimal.
\end{claim}

In fact a stronger statement is true. Note that the following
theorem does not contradict to Conjecture \ref{sej1}, as the
algorithm mentioned there uses the same number of steps as
algorithm W in case $k=2$, $1/\alpha$ is an integer and $\lceil
n/a\rceil\ge 4$.

\begin{theorem} Algorithm W is optimal for any $n$.
\end{theorem}
\begin{proof}
We prove a slightly stronger statement, that algorithm W is optimal even among those algorithms that have access to an unlimited number of extra non-defective elements. This is crucial as we use induction on the number of elements, $n$.

It is easy to check that the answer for a set that is greater than $2a+\delta$ is always {\sc no}, while if both defective elements are in a set of size $2a+\delta$, then the answer is {\sc yes}.
We say that a question is {\em small} if its size is at most $a$,
and {\em big} if its size is between $a+1$ and $2a+\delta$. Note that
small questions test if there is at least one defective element in the
set, while big questions test if both defective elements are in the
set. Suppose by contradiction that there exists an algorithm Z that is
better than W, i.e. there is a set of elements for which Z is
faster than W. Denote by $n$ the size of the smallest such set and
by $z(n)$ the number of steps in algorithm Z. We will establish
through a series of claims that such an $n$ cannot exist. It
already follows from Claim \ref{egyes} that $n$ has to be at least
$3a + \delta + 2^{\lceil\log a\rceil} +1$.

Note that for $n= 3a + \delta + 2^{\lceil\log a\rceil}$ algorithm
W uses $\lceil\log (n-1)\rceil=\lceil\log (2a+\delta -1)\rceil +1$
questions. An important tool is the following lemma.

\begin{lemma}\label{nullas} If $n \ge 3a + \delta + 2^{\lceil\log a\rceil} +1$, then algorithm Z has to start with a
big question. Moreover, it can ask a small question among the
first $z(n)-\lceil \log (2a+\delta-1)\rceil$ questions only if one
of the previous answers was {\sc yes}.
\end{lemma}
\begin{proof}
First we prove that algorithm Z has to start with a big question.
Suppose it starts with a small question. We show that in case the
answer is {\sc no}, it cannot be faster than algorithm W. In this
case after the first answer there are at least $n-a$ (and at most
$n-1$) elements which can be defective, and an unlimited number of
non-defective elements, including those which are elements of the
first question. By induction algorithm W is optimal in this case,
and one can easily see that it cannot be faster if there are more
elements, hence algorithm Z cannot be faster than algorithm W on
$n-a$ elements plus one more question. On the other hand algorithm
W clearly uses this many questions (as it starts with a question
of size $a$), hence it cannot be slower than algorithm Z.

Similarly, to prove the moreover part, suppose that the first
$z(n)-\lceil \log (2a+\delta-1)\rceil$ answers are {\sc no} and
one of these questions, $A$ is small. Let us delete every element
of $A$. By induction algorithm W is optimal on the remaining at
least $n-a$ elements, hence similarly to the previous case,
algorithm Z uses more questions than algorithm W on $n-a$
elements, hence cannot be faster than algorithm W. More precisely,
we can define algorithm W', which starts with asking $A$, and
after that proceeds as algorithm W. One can easily see that
algorithm W' cannot be slower than algorithm Z or faster than
algorithm W.
\qedbox \end{proof}

Note that a {\sc yes} answer would mean that $\lceil \log
(2a+\delta-1)\rceil$ further questions would be enough to find a
defective with binary search, hence in the worst case, (when the
most steps are needed) no such answer occurs among the first
$z(n)-\lceil \log (2a+\delta-1)\rceil$ questions anyway. Now we
can finish the proof of the theorem with the following claim.

\begin{claim}\label{harmas} If $n>3a+\delta+2^{\lceil \log
a\rceil}$, then algorithm W is optimal.
\end{claim}
\begin{proof} If not, then the smallest $n$ for which W is
not optimal must be of the form $2a+\delta+2^{\lceil \log
a\rceil}+za+1$, where $z\ge 1$ integer. (This follows from the
fact that the number of required questions is monotone in $n$ if
we allow the algorithm to have access to an unlimited number of
extra non-defective elements.) By contradiction, suppose that
algorithm Z uses only $\lceil\log(2a+\delta-1)\rceil +z$
questions. Suppose the answer to the first $z$ questions are {\sc
no}. Then by to Lemma \ref{nullas}, these questions are big.
Suppose that the $z+1$st answer is also {\sc no}. We distinguish
two cases depending on the size of the $z+1$st question $A$. In
both cases we will use reasoning similar to the one in Theorem
\ref{approxi}

Case 1. The $z+1$st question is small. After the answer there are
$\lceil\log(2a+\delta-1)\rceil-1$ questions left, so depending on
the answers given to them, any deterministic algorithm can choose
at most $2^{\lceil\log(2a+\delta-1)\rceil-1}$ elements. Hence
algorithm Z gives us after the $z+1$st answer a set $B$ of at most
$2^{\lceil\log(2a+\delta-1)\rceil-1}$ elements, which contains a
defective.

Before starting the algorithm, all the ${n \choose 2}$ pairs are
possible candidates to be the set of defective elements. However,
after the $z+1$st question (knowing the algorithm) the only
candidates are those which intersect $B$. The $z+1$st question
shows at most $a$ non-defective elements, but all the pairs which
intersect neither $A$ nor $B$ have to be excluded by the first $z$
questions. Thus ${n-|A|-|B| \choose 2} \ge{2a+\delta+2^{\lceil
\log a\rceil}+za+1-a-2^{\lceil\log(2a+\delta-1)\rceil-1}\choose
2}\ge {(z+1)a+\delta+1 \choose 2}$ pairs should be excluded, but
$z$ questions can exclude at most $z{2a+\delta \choose 2}$ pairs,
which is less if $z\ge 1$.

Case 2. The $z+1$st question is big. After it we have
$\lceil\log(2a+\delta-1)\rceil-1$ questions left, so depending on
the answers given to them, any deterministic algorithm can choose
at most $2^{\lceil\log(2a+\delta-1)\rceil-1}$ elements. This means
that we have to exclude with the first $z+1$ questions at least
${2a+\delta+2^{\lceil \log a\rceil}+za+1
-2^{\lceil\log(2a+\delta-1)\rceil -1}\choose 2}\ge
{(z+2)a+\delta+1\choose 2}$ pairs. But they can exclude at most
$(z+1){2a+\delta \choose 2}$ pairs, which is less if $z\ge 1$.
\qedbox \end{proof}

This finishes the proof of the theorem.
\qedbox \end{proof}

\section{Open problems}

It is quite natural to think that $g(n,k,\alpha, m)$ is increasing in
$n$ but we did not manage to prove that. The monotonicity in $k$
and $m$ is obvious from the definition.
On the other hand, we could have defined $g(n,k,\alpha,m)$ as the smallest number of questions needed to find $m$ defectives assuming there are \emph{exactly} $k$ defectives (instead of \emph{at least} $k$ defectives) among the $n$ elements, in which case the monotonicity in $k$ is far from trivial. We conjecture
that this definition gives the same function as the original one.

It might seem strange to look for monotonicity in $\alpha$, but we
have seen that for $m=1$ we can reach the information theoretic
lower bound (which is $\lceil \log (n-k+1) \rceil$ in this
setting) for $\alpha\le 2/(n-k+1)$. All the theorems from Section
\ref{mbarmi} also suggest that the smaller $\alpha$ is, the faster
the best algorithm is even for general $m$. Basically in case of a
{\sc no} answer it is better if $\alpha$ is small, and in case of
a {\sc yes} answer the size of $\alpha$ does not matter very much,
since the process can be finished fast. However, we could only
prove Theorem \ref{maintheorem} concerning this matter.

Another interesting question is if we can choose $\alpha$. If
$m=1$ then we should choose $\alpha\le 1/(n-k+1)$, and as we have
mentioned in the previous paragraph, we believe that a small enough
$\alpha$ is the best choice.

Another possibility would be if we were allowed to choose a new
$\alpha$ for every question. Again, we believe that the best
solution is to choose the same, small enough $\alpha$ every time.
This would obviously imply the previous conjecture.

Finally, a more general model to study is the following. We are given two parameters, $\alpha\ge\beta$. If at least an $\alpha$ fraction of the set is defective, then the answer is {\sc yes}, if at most a $\beta$ fraction, then it is {\sc no}, while in between the answer is arbitrary. With these parameters, this paper studied the case $\alpha=\beta$. This model is somewhat similar to the threshold testing model of \cite{damaschke}, where  instead of ratios $\alpha$ and $\beta$ they have fixed values $a$ and $b$ as thresholds.

\end{document}